\documentclass[11 pt]{article}
\usepackage{amssymb}
\usepackage{amsmath}
\usepackage{theorem}
\usepackage{epsfig}
\usepackage{verbatim}
\usepackage{graphicx}
\usepackage{subfigure}
\usepackage{cancel}
\usepackage{epstopdf}
\usepackage{a4wide}
\usepackage{pdflscape}
\usepackage{color}
\usepackage{enumerate}

\def\ep{\epsilon}

\def\Im{{\rm{Im}}\,}

\DeclareMathOperator{\sgn}{sgn}

\newcommand{\R}{{\mathbb{R}}}
\newcommand{\C}{{\mathbb{C}}}
\newcommand{\Z}{{\mathbb{Z}}}
\newcommand{\N}{{\mathbb{N}}}

\newcommand{\T}{{\mathbb{T}}}
\newcommand{\bx}{\mathbf{x}}
\newcommand{\by}{\mathbf{y}}
\newcommand{\bu}{\mathbf{u}}
\newcommand{\be}{\mathbf{e}}
\newcommand{\pa}{\partial}

\newtheorem{df}{Definition}

\newtheorem{lem}{Lemma}

\newtheorem{rem}{Remark}

\newtheorem{thm}{Theorem}

\newenvironment{proof}{\begin{trivlist} \item[] {\em Proof:}}{\hfill $\Box$
                       \end{trivlist}}

\title{The lifespan of classical solutions for the inviscid  Surface Quasi-geostrophic equation.}

\author{\'Angel Castro, Diego C\'ordoba and Fan Zheng 
}

\begin{document}

\maketitle

\begin{abstract}
We consider classical solutions of the inviscid Surface Quasi-geostrophic equation that are
a small perturbation  $\epsilon$ from a radial stationary solution $\theta=|x|$. We use a modified energy method to prove the existence time of classical solutions from $\frac{1}{\epsilon}$ to a time scale of $\frac{1}{\epsilon^4}$. Moreover, by perturbing in a suitable direction we construct global smooth solutions, via bifurcation, that rotate uniformly in time and space.
\end{abstract}

\section{Introduction}
The Surface Quasi-geostrophic equation (SQG) is an active scalar equation
\begin{align*}
\pa_t \theta +\bu\cdot\nabla\theta =0
\end{align*}
for $\theta : \R^2\times\R \to \R$, where the incompressible velocity is related to $\theta$ by $$\bu(\bx,t)=-\nabla^\perp \Lambda^{-1} \theta(\bx,t)$$ and the non-local operator is defined by  $\Lambda=\left(-\Delta\right)^\frac{1}{2}$.

This equation has a geophysical origin (see \cite{Constantin-Majda-Tabak:formation-fronts-qg,Held-Pierrehumbert-Garner-Swanson:sqg-dynamics,Pedlosky:geophysical,Majda-Bertozzi:vorticity-incompressible-flow} for more details) and its mathematical analysis was initially studied by P. Constantin, A. Majda and E. Tabak in \cite{Constantin-Majda-Tabak:formation-fronts-qg} motivated because its similarity with the 3D Euler equations and as a candidate model for a finite time front formation (see also \cite{Const94}).

The local well-posedness of solutions in $H^s$ of the SQG equation is well understood. The standard energy estimates for SQG gives $$\frac{d}{dt}||\theta||_{H^s}^2\leq \left(||\nabla u||_{L^{\infty}}+||\nabla\theta||_{L^\infty}\right) ||\theta||_{H^s}^2.$$ Since the velocity  $u$ is a singular integral operator with respect to the active scalar $\theta$   we can close the a priori estimates for $s> 2$ which yields a local time of existence.
The goal of this paper is to construct solutions of SQG that extend  in time the existence of classical solutions of initial size $\epsilon$  beyond the hyperbolic existence time $O\left(\frac{1}{\epsilon}\right)$.


\subsection{Previous results on long time existence of smooth solutions of SQG}

However very few results are known about the global regularity or long time behaviour of smooth solutions. Global existence of weak solutions in $L^2$ was shown by Resnik in \cite{Resnick:phd-thesis-sqg-chicago} (see also \cite{Marchand:existence-regularity-weak-solutions-sqg, TSV19} for a lower regularity class and \cite{ConstNguy18} in the case of a bounded domain). For higher regularity solutions Dritschel \cite{Dritschel:exact-rotating-solution-sqg} constructed global solutions that have $C^{\frac12}$ regularity. Later, in  \cite{Castro-Cordoba-GomezSerrano:global-smooth-solutions-sqg}, the existence of rotating solutions is proven by Castro et al. with $C^4-$regularty and $3-$fold symmetry, and in \cite{GS19}  Gravejat and Smets showed the existence of travelling waves. In both  \cite{Castro-Cordoba-GomezSerrano:global-smooth-solutions-sqg} and \cite{GS19} the solutions are smooth and have compactly supported $\theta$. In the opposite direction Kiselev and Nazarov, in \cite{KN12}, proved arbitrary bounded growth of high Sobolev norms on finite time intervals in the case of periodic solutions and Friedlander and Shvydkoy
\cite{Friedlander-Shvydkoy:unstable-spectrum-sqg} showed the existence of unstable eigenvalues of the spectrum.

It is an open problem whether the SQG equation, from a smooth initial data, develops finite time singularities or not. Numerical simulations suggested, see \cite{Constantin-Majda-Tabak:formation-fronts-qg}, the possible formation of singularties  with a hyperpolic saddle scenario which later Ohkitani and Yamada \cite{Ohkitani-Yamada:inviscid-limit-sqg} and Constantin et al \cite{Constantin-Nie-Schorghofer:nonsingular-sqg-flow} suggested that the growth was double exponential.  C\'ordoba \cite{Cordoba:nonexistence-hyperbolic-blowup-qg} ruled out a blow-up for this scenario  and bounded the growth by a quadruple exponential. Which was further improved by C\'ordoba and Fefferman \cite{Cordoba-Fefferman:growth-solutions-qg-2d-euler} to a double exponential. Many years later, with bigger computational power and improved algorithms  Constantin et al. \cite{Constantin-Lai-Sharma-Tseng-Wu:new-numerics-sqg},  showed no evidence of the existence of singularities under the same hyperbolic scenario. Moreover,  they observed the depletion of the hyperbolic saddle past the previously computed times.

\subsection{Radially homogenous solutions and main results}

The aim of this paper is the study of the time of existence of certain smooth solutions which are small perturbations of a stationary radial solution $\theta = |\bx|$. These solutions will have a m-fold symmetry (with some $m \geq 3$) of the form $$\theta(\bx,t)=C|\bx|+|\bx|G(\alpha,t),$$ where  $\alpha=\arg(\bx)$  and $G$ is a $2\pi-$periodic function with the symmetry assumption $$G(\alpha,t)=\sum_{|n|\geq 1} G_{mn}(t)e^{i m n  \alpha}.$$

These unbounded solutions were studied by T. Elgindi and I-J. Jeong in \cite{EJ20}, in the case $C=0$. They prove local well-posedness for $G(\alpha,0) \in C^{k,\alpha}$ with $k\geq 0$ and $0<\alpha<1$.

 The radially homogeneous structure of these solutions allows us to obtain a 1D equation for $G$ in the same spirit as in \cite{CaCo10} (see also \cite{EJ20}) where the solutions have the form $$\theta(x_1,x_2,t)= x_2g(x_1,t).$$  Let $f(\alpha,t)=G(\alpha-2ct,t)$ where $c$ is a certain constant. Then $f$ satisfies the following equation
\begin{equation*}
\pa_t f+C\pa_\alpha S f =2Sf\pa_\alpha f -f\pa_\alpha Sf,
\end{equation*}
with
\begin{align*}
&Sf(\alpha)=\int_{0}^{2\pi}S(\alpha-\beta)f(\beta)d\beta,\\
&S(\alpha)=-\frac{1}{8\pi}(1+3\cos(2\alpha))\log(1-\cos(\alpha)).
\end{align*}

It was shown in \cite{EJ20} that the possible existence of finite time singularities for this 1D model with $C=0$ leads to a
singularity formation in the class of Lipschitz solutions with compact support to the SQG equation.

The motivation of our paper is to study the lifespan of these radially homogeneous solutions for $C\neq 0$ and  for a small perturbation of the stationary radial solution $\theta=|\bx|$ $(G=0)$. Our first main result shows that an $\epsilon$ perturbation of $G=0$ gives the existence of a solution of the following equation
\begin{equation}\label{f-eqn}
\pa_t f+\pa_\alpha S f =2Sf\pa_\alpha f -f\pa_\alpha Sf,
\end{equation}
  for a time $T\sim \frac{1}{\ep^4}$. We emphasis that standard methods yield local existence for \eqref{f-eqn} for times $T\sim \frac{1}{\ep}.$

\begin{thm}\label{longexistence}
There is $\ep_0 > 0$ such that if $s \ge 16$, $m\geq 3$,  $f_0 \in H^s\left(\mathbb{T}\right)$ with zero mean and $m-$fold symmetry, i.e., $$f_0\left(\alpha+\frac{2\pi}{m}\right)=f_0\left(\alpha\right)\quad \forall \alpha \in \mathbb{T}$$   and $\|f_0\|_{H^{16}\left(\mathbb{T}\right)} = \ep \le \ep_0$, then there are $T \approx \ep^{-4}$ and a solution $f \in C^k\left([0,T], H^{s-k}\left(\mathbb{T}\right)\right)$, $0 \le k < s - 1/2$, to the equation
\[
\pa_t f+\pa_\alpha S f = N(f) :=2Sf\pa_\alpha f -f\pa_\alpha Sf,
\]
such that $\|f(t)\|_{H^{16}(\mathbb{T}} \lesssim \ep$ for all $t \in [0, T]$.
\end{thm}

The second result of this paper deals with the existence of travelling solutions for equation \eqref{f-eqn} which yields unbounded Lipschitz rotating solutions for SQG.

\begin{thm}\label{thmtravelling}
For each $c > 0$ and integer $m \ge 3$ there is an open interval $I$ containing 0 such that for all $\xi \in I$, there is a $m$-fold symmetric travelling wave solution $f_{m,\xi}$  of the equation \eqref{f-eqn} such that $f_{m,\xi}$ is analytic in the strip $\{\alpha: |\Im\alpha| < c\}$.
\end{thm}

\subsection{Main ideas of the proofs}

In this section we give a brief description of the strategy and main ideas used in the proofs of Theorem 1 and 2.

\subsubsection{Strategy of the proof of Theorem 1: On Normal Forms}

The dispersion is the main mechanism used in Theorem 1 to extend the time of existence of the perturbed stationary radial solution. Similar mechanism was used to study the dynamics of patch-type solutions (i.e. piecewise constant solutions) of SQG. The dynamics of the contour of these patches satisfy  a time reversible quasilinear dispersive equation. Global stability of the half-plane patch stationary solution, under small and localized perturbations, was proven in \cite{CGI} with a more singular SQG velocity. See also \cite{hunter2018global, hunter2020global}  for globally asymptotically stable solutions on different related models describing the dynamics of an SQG patch-type solution. There is, however, a proof in \cite{KRYZ} of finite-time singularities for a patch in the presence of a boundary for a less singular SQG velocity (see also \cite{gancedo2018local}).

The equation for the radially homogeneous solution of the SQG equation can be transformed into a nonlinear dispersive equation
\[
\pa_t f+\pa_\alpha S f = N(f) :=2Sf\pa_\alpha f -f\pa_\alpha Sf,
\]
where $S$ is a Fourier multiplier defined by
\[
\mathcal F(Sf)(n) = \frac{|n|^2 - 1}{|n|^3 - 4|n|}\hat f(n)
\]
and the nonlinearity $N(f)$ is a quasilinear one because it involves taking one derivative of $f$. If one runs energy estimates directly, then the loss of derivatives can be avoided by integrating by parts in space, and we get a lifespan of $\approx \ep^{-1}$ because the nonlinearity is a quadratic one. We can, however, do better by taking advantage of the dispersive effect of the linear part $\pa_t f+\pa_\alpha S f=0$. It was first observed by Poincar\'e in the context of ordinary differential equations (see \cite{Ar}) that if the linear evolution
\[
\pa_t f + Af = 0
\]
is non-resonant, in the sense that for any three eigenvectors $f_1$, $f_2$, $f_3$ of $A$, the corresponding eigenvalues $\lambda_1$, $\lambda_2$ and $\lambda_3$ satisfy the condition
\[
\lambda_1 + \lambda_2 \neq \lambda_3,
\]
then any equation of the form
\[
\pa_t f + Af = Q(f, f)
\]
where $Q(f, f)$ is a quadratic form in $f$, can be transformed into one of the form
\[
\pa_t g + Ag = C(g, g, g)
\]
where $g - f$ is a quadratic form in $f$, and $C(g, g, g)$ is a form at least cubic in $g$. Then the growth of $\|g\|$ can be estimated by
\[
\frac{d}{dt}\|g(t)\| \lesssim \|g(t)\|^3
\]
giving a lifespan of $\approx \ep^{-2}$.

The above process is called the ``normal form transformation".
It was extended to the case of partial differential equations by Shatah \cite{Sh}. In this setting it is sometimes more convenient to reformulate the normal form transformation as integration by parts in time as follows:
Let $i\lambda$ be the multiplier of $A = S\pa_\alpha$, that is,
\[
\mathcal F(\pa_\alpha S f)(n) = i\lambda(n)\hat f(n).
\]
Then the linear evolution is
\[
\hat f(\cdot, t)(n) = e^{-it\lambda(n)}\hat f(\cdot, 0)(n).
\]
Putting this in the right-hand side, the nonlinearity becomes
\begin{align*}
\mathcal F(N(f,t))(n)
&= \sum_{n_1+n_2=n} c_{n_1,n_2}\hat f(n_1,t)\hat f(n_2,t)\\
&= \sum_{n_1+n_2=n} c_{n_1,n_2}e^{-it(\lambda(n_1)+\lambda(n_2))}\hat f(n_1,0)\hat f(n_2,0).
\end{align*}
where $c_{n_1,n_2}$ are constants computable from the expression of $N$.
Hence, to the second order, we have that
\begin{align*}
\hat f(\cdot, t)(n) &= e^{-it\lambda(n)}\hat f(\cdot, 0)(n)\\
&+ \int_0^t \sum_{n_1+n_2=n} c_{n_1,n_2}e^{-it\lambda(n)+is(\lambda(n)-\lambda(n_1)-\lambda(n_2))}\hat f(n_1,0)\hat f(n_2,0)ds.
\end{align*}
If for all $n = n_1 + n_2$ we have that
\[
\lambda(n) \neq \lambda(n_1) + \lambda(n_2)
\]
then we can integrate the right-hand side by parts and get
\begin{align*}
\hat f(\cdot, t)(n) &= e^{-it\lambda(n)}\hat f(\cdot, 0)(n)\\
&+ \sum_{n_1+n_2=n} c_{n_1,n_2}\frac{e^{-it(\lambda(n_1)+\lambda(n_2))}-e^{-it\lambda(n)}}{\lambda(n)-\lambda(n_1)-\lambda(n_2)}\hat f(n_1,0)\hat f(n_2,0)ds
\end{align*}
with an error of the form
\[
\int_0^t O(f(s)^3)ds
\]
yielding a lifespan $\approx \ep^{-2}$.

Going further, if for all $n = n_1 + n_2 + n_3$,
\[
\lambda(n) \neq \lambda(n_1) + \lambda(n_2) + \lambda(n_3)
\]
then one can apply the normal form transformation to obtain an evolution equation whose right-hand side is quartic, and a lifespan $\approx \ep^{-3}$ can be shown. More generally, if the normal form transformation can be iterated $n$ times, then one can prove a lifespan $\approx \ep^{-n-1}$.

Unfortunately in our case, the linear operator is $A = \partial_\alpha S$ satisfies the first non-resonance condition, but fails the second. The failure is mild, however, in the sense that all the tuples $(n_1, n_2, n_3)$ satisfying
\[
\lambda(n) = \lambda(n_1) + \lambda(n_2) + \lambda(n_3)
\]
are degenerate, i.e., $(n_1, n_2, n_3) = (k, -k, l)$ or $(k, l, -k)$ or $(l, k, -k)$. In this case, according to Theorem 4.3 in \cite{KaPo}, the equation can be rewritten in the form
\[
\pa_tg + \pa_\alpha Sg = Q(g, g)g + \text{ terms at least quartic in }g.
\]
where $Q(g, g)$ is a Fourier multiplier whose coefficients depend on $g$,
or, an ``integrable symbol" as defined in Section 5 of \cite{BeFePu},
which usually does not cause trouble in $L^2$-based energy estimates.
This effectively amounts to the second application of the normal form transformation. It happens that the third non-resonance condition is also satisfied, i.e., for all $n = n_1 + n_2 + n_3 + n_4$,
\[
\lambda(n) \neq \lambda(n_1) + \lambda(n_2) + \lambda(n_3) + \lambda(n_4)
\]
and one more iteration of the normal form transformation yields a lifespan $\approx \ep^{-4}$. It remains an interesting question if one more iteration of the normal form transformation is possible, which boils down to a Diaphantine equation whose nontrivial integer solutions seem quite illusive.

\subsubsection{Strategy of the proof of Theorem 2: Bifurcation}

In order to prove theorem \ref{thmtravelling} we look for solutions to the equation \eqref{f-eqn}  of the form $f(\alpha,t)=h(\alpha+vt)$ which yields the equation
\begin{align*}
vh'+Sh'=2Shh'-fSh',
\end{align*}
where now the unknowns   are the speed of the way $v$ and the $2\pi-$periodic function $h$. In order to solve this equation we will bifurcate in the parameter $v$ from $h=0$ using the Crandall-Rabinowitz theorem \cite{Crandall-Rabinowitz:bifurcation-simple-eigenvalues}.

\subsection{Outline of the paper}

We start in section \ref{laecuacion} showing a suitable setting for these solutions that are rotationally symmetric around the origin. Part of this task was already done in \cite{EJ20} by T. Elgindi and I-J. Jeong. For sake of completeness we will give all the details of derivation of equation \eqref{f-eqn} from SQG. In section 3 we analyze the dispersion relation and resonances. In section 4 we introduce some technical tools that will be used in section 5 to prove Theorem \ref{longexistence}. Finally, in section 6, we will show Theorem \ref{thmtravelling}.

\section{The equation of motion}\label{laecuacion}
In this section  we derive the equation for $G(\alpha,t)$ in order to obtain solutions of the form $\theta(\bx,t)=|\bx|+|\bx|G(\alpha,t)$. First of all, we have to understand the operator $\Lambda^{-1}$ acting on this kind of unbounded functions. The part involving the term $|\bx|G(\alpha,t)$ was already consider by T. Elgindi and I-J. Jeong in \cite{EJ20}. We present here all the details of a different derivation for sake of completeness.

The equation of motion is
\begin{align}\label{SQG}
&\partial_t\theta + \bu \cdot \nabla\theta = 0,
&\bu=-\nabla^\perp\psi, && \psi=\Lambda^{-1}\theta.
\end{align}
Here the operator $\Lambda=\sqrt{-\Delta}.$

In $\R^2$ the operator $\Lambda^{-1}$ is given by
\begin{align}\label{rep1}
\Lambda^{-1}\theta(\bx)=\frac{1}{2\pi}\int_{\R^2}\frac{\theta(\by)}{|\bx-\by|}d\by
\end{align}
for functions $\theta(\bx)$ which decay fast enough at the infinity. As explained in the introduction we will study solutions of \eqref{SQG} of the type
\begin{align*}
\theta(\bx,t)=|\bx|+|\bx|G(\alpha,t),
\end{align*}
where $\alpha$ is the argument of $\bx$ and $G(\alpha,t)$ is a real function such that $$G(\alpha,t)=\sum_{|n|\geq 3} G_n(t)e^{in \alpha}.$$ Because the lack of decay at the infinity of these function we can not use the representation \eqref{rep1} for $\Lambda^{-1}$. Instead of that we will use a different representation that we introduce below.

We will use polar coordinates \begin{align*}
&\bx(\rho,\alpha)=\rho(\cos(\alpha),\, \sin(\alpha)), \\
&\be_\rho=(\cos(\alpha),\sin(\alpha)), \quad  \be_\alpha=(-\sin(\alpha),\cos(\alpha)),\\
&\nabla = \be_\rho\pa_\rho +\frac{1}{\rho}\be_\alpha \pa_\alpha.
\end{align*}
We will also use the notation $\overline{f}(\rho,\alpha)=f(\bx(\rho,\alpha))$ for a general function $f: \R^2\to \R$.

Then $\psi(\bx)=\Lambda^{-1}_{\text{new}}\theta(\bx)$ will be given by
\begin{align}\label{rep2}
\overline{\psi}(\rho,\alpha)=\Lambda^{-1}_{\text{new}}\theta(\bx(\rho,\alpha))=\frac{1}{2\pi}P.V.\int_{0}^\infty\int_{0}^{2\pi}k(\rho,s,\alpha,\beta) \overline{\theta}(s,\beta)d\beta sds,
\end{align}
where
\begin{align*}
k(\rho,s,\alpha,\beta)&=\frac{1}{\sqrt{\rho^2+s^2-2\rho s\cos(\alpha-\beta)}}
-\frac{1}{\sqrt{\rho^2+s^2}}-\frac{\rho}{\rho^2+s^2}\cos(\alpha-\beta)\\ &
-\frac{3}{2}\frac{\rho^2}{\left(\rho^2+s^2\right)^\frac{3}{2}}\cos^2(\alpha-\beta)+A\frac{\rho^4}{\left(\rho^2+s^2\right)^k},
\end{align*}
where $P.V.$ means the principal value at the infinity, and $A\in\R$, $k\geq 3$ are suitable constants we will next choose so as to guarantee that
\[
\Lambda_{\text{new}}^{-1}\Lambda_{\text{new}}^{-1}\theta(\bx(\rho,\alpha))
=(-\Delta)^{-1}\theta(\bx(\rho,\alpha)).
\]

We compute the left-hand side. For the kernel we have that
\begin{align*}
&k(\rho,s,\alpha,\beta)=O(s^{-4}),\\
&k(\rho,s,\alpha,\beta)-\frac{\rho^3}{\left(\rho^2+s^2\right)^2}\left(-\frac{1}{2}+\frac{5}{2}\cos^2(\alpha-\beta)\right)\cos(\alpha-\beta)=O(s^{-5})\end{align*} for $s\to \infty$  and then, by using \eqref{rep2},
\begin{align*}
\Lambda^{-1}_{\text{new}}|\bx(\rho,\alpha)|= c_1(A,k) \rho^2,
\end{align*}
with the constant $c_1(A,k)$ given by the absolutely convergent integral
\begin{align*}
c_1(A,k)&=\frac{1}{2\pi}P.V.\int_{0}^\infty \int_{0}^{2\pi} k(1,s,0,\beta)d\beta s^2 ds
=\frac{1}{2\pi}\int_{0}^{2\pi}\int_{0}^\infty k(1,s,0,\beta)s^2 ds d\beta\\
&=d_1+d_2(k)A,\text{ where }d_2(k)=\frac{(2k-5)!!}{4(2k-2)!!}=\frac{1\cdot3\cdots(2k-5)}{4(2\cdot4\cdots(2k-2))}.
\end{align*}
One more application of the operator $\Lambda_{\text{new}}^{-1}$ gives
\begin{align*}
\Lambda^{-1}_{\text{new}}\Lambda^{-1}_{\text{new}}|\bx(\rho,\alpha)|
=c_1(A,k)\Lambda ^{-1}_{\text{new}}|\bx(\rho,\alpha)|^2=c_1(A,k)c_2(A,k)\rho^3,
\end{align*}
with the constant $c_2(A,k)$ given by the absolutely convergent integral
\begin{align*}
c_2(A,k)&=\int_{0}^{2\pi}\frac{1}{2\pi}\int_{0}^\infty \left( k(1,s,0,\beta)-\frac{1}{\left(1+s^2\right)^2}\left(-\frac{1}{2}+\frac{5}{2}\cos^2(\alpha-\beta)\right)\cos(\alpha-\beta)\right)s^3 ds d\beta\\
&=h_1+h_2(k)A,\text{ where }h_2(k)=\frac{1}{4\pi(k-1)(k-2)}.
\end{align*}
Then $p(A,k)\equiv c_1(A,k)c_2(A,k)$ can take values in $[p_{\min},+\infty)$, where $p_{\text{min}}=-\frac{d_1^2}{2}\frac{h_2(k)}{d_2(k)}-\frac{h_1^2}{2}\frac{d_2(k)}{h_2(k)}$. From $\frac{(2k-1)!!}{(2k)!!}>\frac{1}{2}\cdot\frac{(2k-2)!!}{(2k-1)!!}$ we deduce that $\frac{(2k-1)!!}{(2k)!!}>\frac{1}{2\sqrt k}$, so $\frac{h_2(k)}{d_2(k)}=O(1/\sqrt k)$ and by taking $k$ large we can make it as small as we want. Therefore there exist $k^*$ and $A^*$ such that $p(A^*,k^*)=-\frac{1}{9}$. For these values of $A$ and $k$ we have that $$\Lambda^{-2}_{\text{new}}|\bx(\alpha,\rho)|=(-\Delta)^{-1}|\bx(\alpha,\rho)|= -\frac{1}{9}\rho^3.$$

In addition, for $\Lambda^{-1}_{\text{new}}\left(|\bx|G(\alpha,t)\right)$ we have that, by dominated convergence,
\begin{align*}
\Lambda^{-1}_{\text{new}}\left(|\bx|G(\alpha,t)\right)=\Lambda^{-1}_{\text{new}}\left(\lim_{\ep\to 0}e^{-\ep |\bx|^2}|\bx|G(\alpha,t)\right)=\lim_{\ep\to 0}\Lambda^{-1}_{\text{new}}\left(e^{-\ep |\bx|^2}|\bx|G(\alpha,t)\right).
\end{align*}
Similarly
\begin{align*}
\Lambda^{-1}_{\text{new}}\Lambda^{-1}_{\text{new}}\left(|\bx|G(\alpha,t)\right)=\Lambda^{-1}_{\text{new}}\lim_{\ep\to 0}\Lambda^{-1}_{\text{new}}\left(e^{-\ep |\bx|^2}|\bx|G(\alpha,t)\right)=\lim_{\ep\to 0}\Lambda^{-1}_{\text{new}}\Lambda^{-1}_{\text{new}}\left(e^{-\ep |\bx|^2}|\bx|G(\alpha,t)\right)
\end{align*}
Because the first three Fourier modes of $G(\alpha)$ are zero we have that $\Lambda^{-1}_{\text{new}}\left(e^{-\ep |\bx|^2}|\bx|G(\alpha,t)\right)=\Lambda^{-1}\left(e^{-\ep |\bx|^2}|\bx|G(\alpha,t)\right)$. Also the three first modes of $\Lambda^{-1}\left(e^{-\ep |\bx|^2}|\bx|G(\alpha,t)\right)$ will be zero and then
\begin{align*}
\Lambda^{-1}_{\text{new}}\Lambda^{-1}_{\text{new}}\left(e^{-\ep |\bx|^2}|\bx|G(\alpha,t)\right)=\Lambda^{-2}\left(e^{-\ep |\bx|^2}|\bx|G(\alpha,t)\right)
=(-\Delta)^{-1} \left(e^{-\ep |\bx|^2}|\bx|G(\alpha,t)\right).
\end{align*}
And from here we have that
\begin{align*}
\Lambda^{-1}_{\text{new}}\Lambda^{-1}_{\text{new}}\left(|\bx|G(\alpha,t)\right)=(-\Delta)^{-1}\left(|\bx|G(\alpha,t)\right).
\end{align*}

Then for function of the type $\theta(\bx)=|\bx|+|\bx|G(\alpha)$, with $G(\alpha)=\sum_{|n|\geq 3} G_n e^{inx}$, satisfies $\Lambda^{-2}_{\text{new}}\theta(\bx)=(-\Delta)^{-1}\theta(\bx).$ From now on we will remove the subscript ''new'' in the $\Lambda_\text{new}$.

Plugging in \eqref{SQG} the ansatz $\theta(\bx,t)=|\bx|+|\bx|G(\alpha,t)$ we have that
\begin{align*}
&|\bx|\pa_t G(\alpha,t)-\nabla^\perp\Lambda^{-1}|\bx|\cdot\nabla \left(|\bx|G(\alpha,t)\right)-\nabla^\perp \Lambda^{-1}\left(|\bx|G(\alpha,t)\right)\cdot \nabla |\bx|
\\- &\nabla^\perp\Lambda^{-1}\left(|\bx|G(\alpha,t)\right)\cdot \nabla \left(|\bx|G(\alpha,t)\right)=0.
\end{align*}
That in radial coordinates reads
\begin{align*}
&\rho\pa_t G(\alpha,t)-2c \rho e^\perp_\rho\cdot \nabla (\rho G(\alpha,t))-e_r\cdot\nabla^\perp\Lambda^{-1}\left(\rho G(\alpha,t)\right)\\
&= \nabla^\perp \Lambda^{-1}(\rho G(\alpha,t))\cdot \nabla\left(\rho G(\alpha,t)\right),
\end{align*}
that we can write
\begin{align}\label{aux1}
&\rho\pa_t G(\alpha,t)-2c\rho \pa_\alpha G(\alpha,t)+\frac{1}{\rho}\pa_\alpha\Lambda^{-1}\left(\rho G(\alpha,t)\right)=\nonumber\\
&\pa_\rho\Lambda^{-1}(\rho G(\alpha,t))\pa_\alpha G(\alpha,t)-\frac{1}{\rho}\pa_\alpha \Lambda^{-1}(\rho G(\alpha,t)) G(\alpha,t).
\end{align}

Note that $k(\rho,s,\alpha,\beta)$ only depends on $\rho$, $s$ and $\alpha-\beta$, so
\begin{align*}
&\Lambda^{-1}(\rho G(\alpha,t))=\frac{1}{2\pi}\int_{0}^\infty \int_{0}^{2\pi} k(\rho,s,\alpha,\beta)s^2 G(\beta,t) dsd\beta\\&=
 \rho^2\int_{0}^{2\pi}K(\alpha-\beta)G(\beta,t)d\beta\equiv \rho^2 K G(\alpha,t)
\end{align*}
where
\begin{align*}
K(\alpha-\beta)=\frac{1}{2\pi}\int_{0}^\infty k(1,s,\alpha, \beta) s^2 ds.
\end{align*}

We can compute that

\begin{align*}
K(\alpha)=d_2(k)A-\frac{1}{8\pi}(6-2\pi+\cos(\alpha))\cos(\alpha)-\frac{1}{8\pi}(1+3\cos(2\alpha))\log(1-\cos(\alpha)).
\end{align*}
For function $G(\alpha,t)=\sum_{|n|\geq 3} \hat{G}_n e^{in\alpha}$ we have that
\begin{align}\label{soperator}
K G(\alpha,t)=SG(\alpha,t)\equiv\int_{0}^{2\pi}S(\alpha-\beta)G(\beta,t)d\beta
\end{align}
with
\begin{align*}
S(\alpha)=-\frac{1}{8\pi}(1+3\cos(2\alpha))\log(1-\cos(\alpha)).
\end{align*}
Then we can write \eqref{aux1}
\begin{align*}
\pa_t G-2c\pa_\alpha G+\pa_\alpha S G =2SG\pa_\alpha G -G\pa_\alpha SG.
\end{align*}

Let $f(\alpha,t)=G(\alpha-2ct,t)$, which yields \eqref{f-eqn}.

\section{Dispersion relation}

In this section we compute de Fourier transform of the function
\begin{align*}
S(\alpha)=-\frac{1}{8\pi}(1+3\cos(2\alpha))\log(1-\cos(\alpha)).
\end{align*}

We will use as definition of the Fourier transform, for a $2\pi-$periodic function $f$,
\begin{align*}
\hat{f}_n=\frac{1}{2\pi}\langle f, e^{-inx}\rangle =\frac{1}{2\pi}\int_{-\pi}^\pi f(x)e^{-inx}dx.
\end{align*}
Thus $$f(x)=\sum_{n=-\infty}^\infty \hat{f}_n e^{inx}.$$
In addition, if $f$ is a real and even function
$$f=\hat{f}_0+2\sum_{n=1}^\infty \hat{f}_n \cos(nx).$$

By Gradshteyn and Ryzhik 1.441 (2),
\begin{align*}
\ln(1 - \cos\alpha) = -2\sum_{n\ge1} \frac{\cos n\alpha}{n} - \ln 2
\end{align*}
so
\begin{align*}
&(3\cos2\alpha +1 )\ln(1 - \cos\alpha)
= -\sum_{n\ge1} \frac{6\cos2\alpha + 2}{n}\cos n\alpha - (3\cos2\alpha + 1)\ln2\\
&= -\sum_{n\ge1} \frac{3(\cos(n + 2)\alpha + \cos(n - 2)\alpha)}{n}
- \sum_{n\ge1} \frac{2\cos n\alpha}{n}- (3\cos2\alpha + 1)\ln2\\
&=-6\cos\alpha - \left( \frac{7}{4} + 3\ln2 \right)\cos2\alpha
- \frac{3}{2} - \ln2 - \sum_{n\ge3} \left( \frac{3}{n - 2} + \frac{3}{n + 2} + \frac{2}{n} \right)\cos n\alpha\\
&=-6\cos\alpha - \left( \frac{7}{4} + 3\ln2 \right)\cos2\alpha
- \frac{3}{2} - \ln2 - \sum_{n\ge3} \frac{8(n^2 - 1)}{n^3 - 4n}\cos n\alpha
\end{align*}
Therefore, for  $n\geq 3$
\begin{align*}
\left((3\cos2\alpha + 1)\ln(1 - \cos\alpha)\right)^{\widehat{}}_n
=-4\frac{|n|^2-1}{|n|^3-4|n|}
\end{align*}
and
\begin{align*}
\hat{S}_n=\frac{1}{2\pi}\frac{|n|^2-1}{|n|^3-4|n|},
\end{align*}
which behaves like $ \frac{1}{2\pi |n|}$, for $|n|\to \infty$.

We summary this section in the following lemma
\begin{lem}\label{multiplier} Let $S$ be the operator given by \eqref{soperator} and $f$ a $2\pi$-periodic  $C^1$-function. Then, for $|n|\geq 3$,
\begin{align}\label{fsf}
\widehat{Sf}_n=\frac{|n|^2-1}{|n|^3-4|n|},
\end{align}
and
\begin{align*}
\widehat{\pa_\alpha Sf}=i\lambda(n),
\end{align*}
with $$\lambda(n)=\left(1+\frac{3}{n^2-4}\right)\sgn n.$$
\end{lem}
\begin{proof} Notice that $\widehat{f\circ g}_n=2\pi \hat{f}_n\hat{g}_n$.
\end{proof}

\begin{rem} We notice that $G$ in \cite{EJ20}-section 4.2 and $Sf$ in this paper must be the same. In \cite{EJ20}-Lemma 4.4 is obtained
$$\hat{G}_k=\frac{1}{-|k|-\frac{3|k|}{k^2-1}}\hat{f}_k,\quad |k|\geq 2,$$
which does not agree with \eqref{fsf}. However in \cite{EJ20}, in the proof of Lemma 4.4, is used that
\begin{align*}
\frac{1}{2\pi}\int \frac{G(\theta)-G(\overline{\theta})}{2\sin^2\left(\frac{\theta-\overline{\theta}}{2}\right)}d\overline{\theta}=-\Lambda \theta.
\end{align*}
Unfortunately, the minus sign is wrong in that formula. In order to check it one can notice that $\Lambda G$ at the maximum of $G$ must be positive. With the correct sign Elgindi and Jeong obtain
\begin{align*}\underbrace{|k|}_{\text{from $\Lambda$ with the correct sign}}+\underbrace{\frac{-3|k|}{k^2-1}}_{\text{from the last integral in pag. 40 divided by $2\pi$}}\hat{G}_k=\hat{f}_k
\end{align*}
which agrees with \eqref{fsf}.
\end{rem}

\subsection{Analysis of resonances}

In this section we shall study the resonances of the linear part of equation \eqref{f-eqn} which is given by $\pa_\alpha Sf$. Here $S$ is the operator in \eqref{soperator} whose multiplier can be found in lemma \ref{multiplier}. Indeed we  will show that there are no 3, 4 or 5 wave resonant interactions, except 4 wave degenerate interactions.
\begin{lem}\label{denom-bound}
Let $n_1,\dots,n_p \in \Z$ and each $n_j$ satisfy $|n_j| \ge 3$. Let $\lambda(n)$ as in lemma \ref{multiplier}. The following hold:

\begin{enumerate}[(i)]

\item  If $p = 3, 5$ and $n_1 + \cdots + n_p = 0$ then
\[
|\lambda(n_1) + \cdots + \lambda(n_p)| \gtrsim 1.
\]

\item  If $p = 4$ and the tuple $(n_1, n_2, n_3, n_4)$ is not totally degenerate, i.e., is not a permutation of $(k, -k, l, -l)$ for some $k$, $l \in \Z$, then
\[
|\lambda(n_1) + \cdots + \lambda(n_4)| \gtrsim\min(|n_1|,\dots,|n_4|)^{-4}.
\]
\end{enumerate}
\end{lem}
\begin{proof}
(i) If $p = 3$, without loss of generality assume $n_1$ and $n_2 > 0$. Then
\[
\lambda(n_1) + \lambda(n_2) + \lambda(n_3)
= \lambda(n_1) + \lambda(n_2) - \lambda(n_1 + n_2)
> 2 - \frac{8}{5} = \frac{2}{5}.
\]

If $p = 5$, without loss of generality we can assume $n_1$, $n_2$, $n_3 > 0$ and $n_4 \ge n_5$. If $n_4 > 0$ then
\[
\lambda(n_1) + \lambda(n_2) + \lambda(n_3) + \lambda(n_4) + \lambda(n_5)
> 4 - \frac{8}{5} > 2.
\]
If $n_5 \le n_4 < 0$ then $n_4 + n_5 = -n_1 - n_2 - n_3 \le -9$, so $n_5 \le -5$ and
\[
\lambda(n_1) + \lambda(n_2) + \lambda(n_3) + \lambda(n_4) + \lambda(n_5)
> 3 - \frac{8}{5} - \frac{8}{7} = \frac{9}{35}.
\]

(ii) Without loss of generality we assume $n_1$, $n_2 > 0$ and $n_3 \ge n_4$. If $n_3 > 0$ then
\[
\lambda(n_1) + \lambda(n_2) + \lambda(n_3) + \lambda(n_4)
> 3 - \frac{8}{5} > 1.
\]
If $n_4 \le n_3 < 0$ then we can further assume $|n_2| \le |n_3| \le |n_4| \le |n_1|$. In this range
\[
\lambda(n_1) + \lambda(n_2) + \lambda(n_3) + \lambda(n_4)
= \int_{n_2}^{-n_4} \int_0^{n_1+n_4} \lambda''(x + y)dydx.
\]
If $n_2 + n_3 = 0$ then $n_3 = -n_2$ and $n_4 = -n_1$, and we get totally degenerate interactions. Otherwise $-n_4 \ge -n_3 \ge n_2 + 1$ and $n_1 + n_4 \ge 1$, so using $\lambda''(x) \gtrsim 1/x^4$ for $x > 2$ we get
\[
\lambda(n_1) + \lambda(n_2) + \lambda(n_3) + \lambda(n_4)
\ge \int_{n_2}^{n_2+1} \int_0^1 \lambda''(x + y)dydx \gtrsim \frac{1}{n_2^4}.
\]
\end{proof}

%
\section{Paradifferential operators and remainders}

To bound the multilinear terms arising from iterated normal form transformation, we define some multilinear operators that we will use frequently.
Let $\Delta_x$ denote the difference in the $x$ variable, i.e., $\Delta_x f(x, \dots) = f(x + 1, \dots) - f(x, \dots)$, $\Delta_x^m$ be $\Delta_x$ iterated $m$ times, and $|_{x=a}$ denote evaluation at $x = a$.

\begin{df}[Multilinear forms]\label{mf}
For $p \ge 3$, $\mu \in \R$ and an increasing function $\nu: \N \to \R$,
let $M_p^{\mu.\nu}$ be the space of $p$-linear maps $M: C^\infty(\T)^p \to \C$ of the form
\[
M(u_1,\dots,u_p) = \sum_{n_1+\cdots+n_p=0,|n_j|\ge3} m(n_1,\dots,n_p)\hat u_1(n_1)\cdots\hat u_p(n_p)
\]
satisfying, for $|n_j| \ge 3$, $N$ being the largest among $|n_1|,\dots,|n_p|$, and $n$ being the third largest,

\begin{enumerate}[(i)]

\item $|m(n_1,\dots,n_p)| \lesssim N^\mu n^{\nu(0)}$, and

\item  For $m \in \N^+$ there is $\delta_m > 0$ such that if $n < \delta_mN$ and $n < |n_j|$, $|n_k|$ then
\[
|\Delta_x^m|_{x=0}m(n_1,\dots,n_j + x,\dots,n_k - x,\dots,n_p)|
\lesssim_m N^{\mu-m}n^{\nu(m)}
\]

\end{enumerate}

Let
\[
M_{p\pm}^{\mu,\nu} = \{M \in M_p^{\mu,\nu}: m(-n_1,\dots,-n_p) = \pm m(n_1,\dots,n_p)\}.
\]
\end{df}

Clearly, for $\alpha \ge 0$, $M_{p(\pm)}^{\mu,\nu} \subset M_{p(\pm)}^{\mu+\alpha,\nu-\alpha} \subset M_{p(\pm)}^{\mu+\alpha,\nu}$,
with the convention that it includes three versions, one for $M_p^{\mu,\nu}$, one for $M_{p+}^{\mu,\nu}$ and one for $M_{p-}^{\mu,\mu}$.

\begin{lem}\label{mult-bound}
Given $M \in M_p^{2\mu,\nu}$, for $\mu \ge s > \nu(0) + 1/2$ and $u \in H^\mu(\T)$ we have
\[
|M(u,\dots,u)| \lesssim_s \|u\|_{H^s}^{p-2}\|u\|_{H^\mu}^2.
\]
\end{lem}
\begin{proof}
Without loss of generality we assume that $m$ and $\hat u$ take values in $\R_{\ge0}$. Let $(n_1,\dots,n_p) \in \Z^p$, $|n_j| \ge 3$, and without loos of generality assume $|n_1| = N$. Then $|\sum_{k=2}^p n_k| = N$, so we can further assume $|n_2| \gtrsim N$. In general we can find $|n_j| = N$, $|n_k| \gtrsim N$ and $|n_l| = n$, so
\[
m(n_1,\dots,n_p) \lesssim N^{2\mu}n^{\nu(0)} \lesssim
\sum_{1\le j<k\le p} \sum_{l=1\atop l\neq j,k}^p |n_j|^\mu|n_k|^\mu|n_l|^{\nu(0)}.
\]
When $u_1 = \cdots = u_p = u$, by symmetry we have
\begin{align*}
|M(u,\dots,u)|&\lesssim \sum_{n_1+\cdots+n_p=0,|n_j|\ge3} |n_1|^\mu|n_2|^\mu|n_3|^{\nu(0)}\hat u(n_1)\cdots\hat u(n_p)\\
&= \int_0^{2\pi} (|\nabla|^\mu u(x))^2(|\nabla|^{\nu(0)}u(x))u(x)^{p-3}\\
&\le \||\nabla|^\mu u\|_{L^2}^2\||\nabla|^{\nu(0)}u\|_{L^\infty}\|u\|_{L^\infty}^{p-3}\\
&\lesssim_s \|u\|_{H^\mu}^2\|u\|_{H^s}^{p-2}
\end{align*}
thanks to Sobolev embedding and the assumption that $s > \nu(0) + 1/2$.
\end{proof}

\subsection{The normal form transformation}
In this section we describe the normal form transformation for multilinear forms that will be used frequently later. For $M \in M_p^{\mu,\nu}$ with
\[
M(u_1,\dots,u_p) = \sum_{n_1+\cdots+n_p=0,|n_j|\ge3} m(n_1,\dots,n_p)\hat u_1(n_1)\cdots\hat u_p(n_p)
\]
we define
\[
L(M)(u_1,\dots,u_p) = \sum_{n_1+\cdots+n_p=0,|n_j|\ge3}
\frac{m(n_1,\dots,n_p)}{\lambda(n_1) + \cdots + \lambda(n_p)}\hat u_1(n_1)\cdots\hat u_p(n_p).
\]
Since the multiplier of $\partial_\alpha S$ is $i\lambda$, we have that
\[
\sum_{j=1}^p L(M)(u_1,\dots,-\partial_\alpha Su_j,\dots,u_p) = -iM(u_1,\dots,u_p)
\]
so
\[
\frac{d}{dt}iL(M)(f,\dots,f) = M(f,\dots,f) + \sum_{j=1}^p iL(M)(f,\dots,\underbrace{2f_\alpha Sf - fSf_\alpha}_{j\text{-th entry}},f).
\]

\begin{lem}\label{normal-form-35} Let $M\in M_{p}^{\mu,\nu}$. Then, for $p=3$, $5$, $$M \in M_{p(\pm)}^{\mu,\nu} \implies L(M) \in M_{p(\mp)}^{\mu,\nu}.$$
\end{lem}
\begin{proof}
To check condition (i) in definition \ref{mf}, we use Lemma \ref{denom-bound} (i) to deduce that
\[
\left|\frac{m(n_1,\dots,n_p)}{\lambda(n_1) + \cdots + \lambda(n_p)}\right|
\lesssim |m(n_1,\dots,n_p)| \lesssim N^\mu.
\]

Now we check condition (ii) in definition \ref{mf}. Assume $m \in \N^+$, $n < \delta_m'N$ for some $\delta_m' > 0$ depending on $\delta_1,\dots,\delta_m$, and $n < |n_j|$, $|n_k|$. Let
\begin{align*}
A(x) &= m(n_1,\dots,n_j + x,\dots,n_k - x,\dots,n_p),\\
B(x) &= \lambda(n_1) + \cdots + \lambda(n_j + x) + \cdots + \lambda(n_k - x) + \cdots + \lambda(n_p).
\end{align*}
Then by binomial expansion of the difference,
\[
\Delta_x^m|_{x=0} \frac{A(x)}{B(x)}
= \sum_{l=0}^m {m \choose l} \Delta_x^{m-l}|_{x=l} A(x) \Delta_x^l|_{x=0} \frac{1}{B(x)}.
\]
If $\delta_m'$ is small enough, then for $x \in [0, m]$ and $i = 1, \dots, m$ we have $\delta_i|n_j + x|$ and $\delta_i|n_k - x| > n$.
Since $M \in M_p^{\mu,\nu}$,
\[
|\Delta_x^{m-l}|_{x=l} A(x)| \lesssim_m N^{\mu-m+l}n^{\nu(m-l)}.
\]
By the fundamental theorem of calculus,
\[
|\Delta_x^l|_{x=0} B(x)^{-1}| \le \sup_{x\in[0,l]} |(B(x)^{-1})^{(l)}|.
\]
When $x \in [0, m]$, $n_j + x$ and $n_k - x \gtrsim N$, so by Lemma \ref{denom-bound} (i), $|B(x)| \gtrsim 1$ and for $l \ge 1$, $|B^{(l)}(x)| \lesssim_l N^{-l-2}$. The same bound holds for the left-hand side, so
\[
\left| \Delta_x^m|_{x=0} \frac{A(x)}{B(x)} \right|
\lesssim_m \sum_{l=0}^m N^{\mu-m+l}n^{\nu(m-l)} N^{-l}
\lesssim_m N^{\mu-m}n^{\nu(m)}.
\]
Since $\lambda$ is an odd function, division by $\lambda(n_1) + \cdots + \lambda(n_p)$ flips the parity.
\end{proof}

When the number of variables, $p$, is even, for $M \in M_p^{\mu,\nu}$ we also define
\[
P(M)(u_1,\dots,u_p) = \sum_{(n_1,\dots,n_p)\in D_p,|n_j|\ge3}
m(n_1,\dots,n_p)\hat u_1(n_1)\cdots\hat u_p(n_p),
\]
where the summation is over the set of totally degenerate tuples
\begin{align*}
D_p = \{(n_1,\dots,n_p): &\exists\text{ a fixed-point free involution }\sigma \in S_p\text{ such that }\\
&n_{\sigma(j)} = -n_j, j=1, \dots, p\}.
\end{align*}

\begin{lem}\label{normal-form-4}Let $M\in M_{p}^{\mu,\nu}$. Then

\begin{enumerate}[(i)]

\item $P: M_{p(\pm)}^{\mu,\nu} \to M_{p(\pm)}^{\mu,\nu}$ is a linear projection.

\item  $M \in M_{p-}^{\mu,\nu} \implies P(M)(f,\dots,f) = 0$.

\item $M \in M_{4(\pm)}^{\mu,\nu} \cap \ker P \implies L(M) \in M_{4(\mp)}^{\mu,\nu'}$, $\nu'(m) = \max_{l=0}^m (\nu(m - l) + 4l + 4)$.

\end{enumerate}

\end{lem}
\begin{proof}
(i) We first check that $M\in M_p^{\mu,\nu}\implies P(M)\in M_p^{\mu,\nu}$.
Since the sum is restricted to $D_p$, the bound on $m$ is trivial.
To show the bound on the derivatives of $m$, note that if $(n_1,\dots,n_p) \in D_p$, $\delta_m < 1$, $n < \delta_mN$ and $n < |n_j|$, $|n_k|$,
then $\sigma$ swaps $j$ and $k$, which implies that $n_k = -n_j$.
Moverover, for all $x \in \N$ we have $n_k - x = -(n_j + x)$,
so $(n_1,\dots,n_j + x,\dots,n_k - x,\dots,n_p) \in D_p$,
and the iterated differences in $x$ are unchanged.
Hence the bound on the iterated differences of $m$ remains true.
The persistence of parity is trivial.

(ii) Since $D_p$ is invariant under the map $(n_1,\dots,n_p) \mapsto (-n_1,\dots,-n_p)$,
\begin{align*}
P(M)(f,\dots,f) = \frac{1}{2}\sum_{(n_1,\dots,n_p)\in D_p,|n_j|\ge3}
&(m(n_1,\dots,n_p)\hat f(n_1)\cdots\hat f(n_p)\\
+ &m(-n_1,\dots,-n_p)\hat f(-n_1)\cdots\hat f(-n_p)).
\end{align*}
Since $M \in M_{p-}^{\mu,\nu}$,
\[
m(-n_1,\dots,-n_p) = -m(n_1,\dots,n_p).
\]
Since $(n_1,\dots,n_p)\in D_p$, there is $\sigma \in S_p$ such that
$n_{\sigma(j)} = -n_j$, so
\[
\hat f(n_1)\cdots\hat f(n_p) = \hat f(n_{\sigma(1)})\cdots\hat f(n_{\sigma(p)}) = \hat f(-n_1)\cdots\hat f(-n_p).
\]
Hence the two terms in the sum cancel, so the whole sum vanishes.

(iii) By Lemma \ref{denom-bound} (ii), for nondegenerate tuples $(n_1,\dots,n_4)$,
\[
\left| \frac{m(n_1,\dots,n_4)}{\lambda(n_1) + \cdots + \lambda(n_4)} \right|
\lesssim m(n_1,\dots,n_4)\min(|n_1|,\dots,|n_4|)^4.
\]
Now we assume $m \in \N^+$, $n < \delta_m'N$ for some $\delta_m' > 0$ depending on $\delta_1,\dots,\delta_m$, and $n < |n_j|$, $|n_k|$.
Define $A(x)$, $B(x)$ and expand the difference of $A(x)/B(x)$ binomially as before. Then
\[
\Delta_x^{m-l}|_{x=l} A(x)\lesssim_m N^{\mu-m+l}n^{\nu(m-l)}\text{ and }
|\Delta_x^l|_{x=0} B(x)^{-1}| \le \sup_{x\in[0,l]} |(B(x)^{-1})^{(l)}|
\]
as before. Since $|B(x)| \gtrsim n^{-4}$ and for $l \ge 1$, $|B^{(l)}(x)| \lesssim_l N^{-l-2}$, it follows that
\[
\left| \Delta_x^m|_{x=0} \frac{A(x)}{B(x)} \right|
\lesssim_m \sum_{m=0}^l N^{\mu-m+l}n^{\nu(m-l)} N^{-l}n^{4l+4}
\lesssim_m N^{\mu-m}n^{\nu'(m)}
\]
if we let $\nu'(m) = \max_{l=0}^m (\nu(m - l) + 4l + 4)$.
\end{proof}

We also need operators on multilinear forms to track the nonlinearity.
For $M \in M_p^{\mu,\nu}$ we define
\begin{align*}
N_1(M)(u_1,\dots,u_{p+1}) &= \sum_{j=1}^p M(u_1,\dots,u_j\partial_\alpha Su_{j+1},\dots,u_{p+1}),\\
N_2(M)(u_1,\dots,u_{p+1})&=\frac{1}{(p+1)!}\sum_{\sigma\in S_{p+1}}
\sum_{j=1}^p M(u_{\sigma(1)},\dots,\partial_\alpha u_{\sigma(j)}Su_{\sigma(j+1)},\dots,u_{\sigma(p+1)}).
\end{align*}

\begin{lem}\label{non-lin-bound} Let $M\in M_{p}^{\mu,\nu}$. Then
\begin{enumerate}[(a)]
\item $M \in M_{p(\pm)}^{\mu,\nu} \implies N_1(M) \in M_{p+1(\mp)}^{\mu,\nu}.$
\item $N_2(M) \in M_{p+1(\mp)}^{\mu,\nu'}$, where $\nu'(m) = \nu(m + 1) + 1$.
\end{enumerate}
\end{lem}
\begin{proof}
(a) $N_1(M)$.
\begin{align*}
N_1(M)(u_1,\dots,u_{p+1}) &= \sum_{n_1+\cdots+n_{p+1}=0,|n_j|\ge3}
\sum_{j=1}^p &m(n_1,\dots,n_j+n_{j+1},\dots,n_{p+1})\\
\times &i\lambda(n_{j+1})\hat u(n_1)\cdots\hat u(n_{p+1}).
\end{align*}
Let $N$ (resp. $N'$) be the largest among $|n_1|,\dots,|n_{p+1}|$ (resp. $|n_1|,\dots,|n_j+n_{j+1}|,\dots,|n_p|$), and $n$ (resp. $n'$) is the third largest. Then $N' \lesssim N$, $n' \lesssim n$. Condition (i) follows from the bound
\[
|m(n_1,\dots,n_j+n_{j+1},\dots,n_{p+1})\lambda(n_{j+1})|
\lesssim N'^\mu n'^{\nu(0)} \lesssim N^\mu n^{\nu(0)}.
\]

To check condition (ii) we assume $m \in \N^+$, $n < \delta_m'N$ for some $\delta_m' > 0$ depending on $\delta_1,\dots,\delta_m$, and distinguish several cases.

{\bf Case 1:} $|n_{j+1}| \le n$. Then the difference acts on the $m$ factor. Since $\lambda$ is bounded, if $\delta_m'$ is small enough then $|$the $m$-th difference$| \lesssim_m N^{\mu-m}n^{\nu(m)}$.

{\bf Case 2:} $n < |n_j|$, $|n_{j+1}|$. Then the difference acts on the $\lambda$ factor. If $\delta_m'$ is small enough, then $|\Delta^m\lambda(n_{j+1})| \lesssim_m |n_{j+1}|^{-m-2} \lesssim_m N^{-m-2}$, so $|$the $m$-th difference$| \lesssim_m N^{\mu-m-2}n^{\nu(0)} \le N^{\mu-m}n^{\nu(m)}$.

{\bf Case 3:} $|n_j| \le n < |n_{j+1}|$, $|n_k|$. Binomial expansion of the differnece gives
\begin{align*}
&\Delta_x^m|_{x=0}(m(n_1,\dots,n_j+n_{j+1} \pm x,\dots,n_k \mp x,\dots,n_{p+1})\lambda(n_{j+1} \pm x))\\
= &\sum_{l=0}^m {m \choose l}\Delta_x^{m-l}|_{x=l}m(n_1,\dots,n_j+n_{j+1} \pm x,\dots,n_k \mp x,\dots,n_{p+1})\\
\times &\Delta_x^l|_{x=0}\lambda(n_{j+1} \pm x).
\end{align*}
If $\delta_m'$ is small enough then $|$the first factor$| \lesssim_m N^{\mu-m+l}n^{\nu(m-l)}$ and, by Case 1, $|$the second factor$| \lesssim_l |n_{j+1}|^{-l} \lesssim_l N^{-l}$, so $|$the above$| \lesssim_m N^{\mu-m}n^{\nu(m)}$.

(b) $N_2(M)$.
\begin{align*}
&N_2(M)(u_1,\dots,u_{p+1})\\
= &\frac{1}{(p+1)!}\sum_{\sigma\in S_{p+1}}
\sum_{n_1+\cdots+n_{p+1}=0,|n_j|\ge3}
\sum_{j=1}^p m(n_{\sigma(1)},\dots,n_{\sigma(j)}+n_{\sigma(j+1)},\dots,n_{\sigma(p+1)})\\
\times &in_{\sigma(j)}s(n_{\sigma(j+1)})\hat u(n_1)\cdots\hat u(n_{p+1}),
\end{align*}
where $s(n) = \lambda(n)/n \sim 1/|n|$ as $|n| \to \infty$. Without loss of generality we assume $|n_1|\ge\cdots\ge|n_{p+1}|$.

We first check condition (i).
If $\sigma(j) \ge 3$ then $|n_{\sigma(j)}| \le |n_3| \le n$,
so
\[
|m(n_{\sigma(1)},\dots,n_{\sigma(j)}+n_{\sigma(j+1)},\dots,n_{\sigma(p+1)})n_{\sigma(j)}s(n_{\sigma(j+1)})| \lesssim N^\mu n^{\nu(0)+1}.
\]
If $\{\sigma(j), \sigma(j + 1)\} = \{1, 2\}$ then $|n_{\sigma(j)}s(n_{\sigma(j+1)})| \lesssim |n_{\sigma(j)}/n_{\sigma(j+1)}|\le p$, so
\[
|m(n_{\sigma(1)},\dots,n_{\sigma(j)}+n_{\sigma(j+1)},\dots,n_{\sigma(p+1)})n_{\sigma(j)}s(n_{\sigma(j+1)})| \lesssim N^\mu n^{\nu(0)}.
\]

Now we assume $\sigma(j) \le 2$ and $\sigma(j + 1) > 3$. If $n \ge \delta_0'N$ (for some $\delta_0' > 0$ depending on $\delta_1$) then
\[
|m(n_{\sigma(1)},\dots,n_{\sigma(j)}+n_{\sigma(j+1)},\dots,n_{\sigma(p+1)})n_{\sigma(j)}s(n_{\sigma(j+1)})| \lesssim N^\mu n^{\nu(0)+1}/\delta_0'.
\]
Now we assume $n < \delta_0'N$. We pair the terms with $\sigma(j) = 1$ and the terms with $\sigma(j) = 2$ as follows:
\begin{equation}\label{quasi-lin-mult}
\begin{aligned}
&\sum_{j=1}^p \sum_{\sigma\in S_{p+1}\atop\sigma(j)\le2,\sigma(j+1)\ge3}
m(n_{\sigma(1)},\dots,n_{\sigma(j)}+n_{\sigma(j+1)},\dots,n_{\sigma(p+1)})
n_{\sigma(j)}s(n_{\sigma(j+1)})\\
= \frac{1}{2}&\sum_{j=1}^p \sum_{\sigma\in S_{p+1}\atop\sigma(j)\le2,\sigma(j+1)\ge3}
s(n_{\sigma(j+1)})(\underbrace{m(n_{\sigma(1)},\dots,n_1+n_{\sigma(j+1)},\dots,n_2,\dots,n_{\sigma(p+1)})}_A n_1\\
+ &\underbrace{m(n_{\sigma(1)},\dots,n_1,\dots,n_2 + n_{\sigma(j+1)},\dots,n_{\sigma(p+1)})}_B n_2).
\end{aligned}
\end{equation}
Both $|A|$ and $|B| \lesssim N^\mu n^{\nu(0)}$ while their difference is
\begin{equation}\label{A-B}
A - B = \sum_{l=0}^{n_{\sigma(j+1)}-1} \Delta_x|_{x=l}
m(n_{\sigma(1)},\dots,n_1 + x,\dots,n_2 + n_{\sigma(j+1)} - x,\dots,n_{\sigma(p+1)}).
\end{equation}
If $\delta_0'$ is small enough (depending on $\delta_1$),
$|$all summands$| \lesssim N^{\mu-1}n^{\nu(1)}$.
Since $\sigma(j + 1) \ge 3$, $|n_{\sigma(j+1)}| \le n$,
so $|A - B| \lesssim N^{\mu-1}n^{\nu(1)+1}$. Also note that
$|n_1 + n_2| = |n_3 + \cdots + n_{p+1}| < pn$.
Then the summands in (\ref{quasi-lin-mult}) are
\begin{align*}
s(n_{\sigma(j+1)})(An_1 + Bn_2)
&= s(n_{\sigma(j+1)})(B(n_1 + n_2) + (A - B)n_1)\\
\text{whose absolute value}
&\lesssim N^\mu n^{\nu(0)+1}+N^{\mu-1}n^{\nu(1)+1}N \lesssim N^\mu n^{\nu(1)+1}.
\end{align*}

Now we check condition (ii). Again assume $|n_1| \ge \cdots \ge |n_{p+1}|$,
$m \in \N^+$ and $n < \delta_m'N$ for some $\delta_m' > 0$ depending on $\delta_1,\dots,\delta_{m+1}$. From the ordering it follows that $|n_1| \ge |n_2| > n \ge |n_3|$. Again we distinguish several cases.

{\bf Case 1:} $\{\sigma(j), \sigma(j + 1)\} = \{1, 2\}$. Then the difference acts on the factor $n_{\sigma(j)}s(n_{\sigma(j+1)})$.
If $\delta_m'$ is small enough, then $|\Delta_x^m($this factor$)| \lesssim_m N^{-m}$, so $|$the $m$-th difference$| \lesssim_m N^{\mu-m}n^{\nu(0)} \le N^{\mu-m}n^{\nu(m)}$.

{\bf Case 2:} $\sigma(j)$, $\sigma(j + 1) \ge 3$ and $\sigma(k)$, $\sigma(l) \le 2$. Then the difference acts on the $m$ factor,
with a bound of $O_m(N^{\mu-m}n^{\nu(m)})$.
If $\delta_m'$ is small enough then $|n_{\sigma(j)}s(n_{\sigma(j+1)})| \lesssim |n_{\sigma(j)}/n_{\sigma(j+1)}|\le p$, so $|$the product$|\lesssim_m N^{\mu-m}n^{\nu(m)}$.

{\bf Case 3:} $\sigma(j) \ge 3$ and $\sigma(j + 1) \le 2$ (assumed to be 1 without loss of generality). Then
\begin{align*}
&\Delta_x^m|_{x=0}(m(n_{\sigma(1)},\dots,n_{\sigma(j)}+n_1 \pm x,\dots,n_2 \mp x,\dots,n_{\sigma(p+1)})n_{\sigma(j)}s(n_1 \mp x))\\
= &n_{\sigma(j)}\sum_{m=0}^l \Delta_x^{m-l}|_{x=l} m(n_{\sigma(1)},\dots,n_{\sigma(j)}+n_1 \pm x,\dots,n_2 \mp x,\dots,n_{\sigma(p+1)})\\
\times &\Delta_x^l|_{x=0}s(n_1 \mp x).
\end{align*}
The first factor is bounded by $N$. If $\delta_m'$ is small enough,
$|$the second one$| \lesssim_m N^{\mu-m+l}n^{\nu(m-l)} \le N^{\mu-m+l}n^{\nu(m)}$ and $|$the third one$| \lesssim_l |n_1|^{-l-1} \lesssim_l N^{-l-1}$, so $|$the above$| \lesssim_m N^{\mu-m}n^{\nu(m)}$.

{\bf Case 4:} $\sigma(j) \le 2$ and $\sigma(j + 1) \ge 3$.
As before we pair the terms
\begin{align*}
&s(n_{\sigma(j+1)})\Delta_x^m|_{x=0}(\underbrace{m(n_{\sigma(1)},\dots,n_1 \pm x+n_{\sigma(j+1)},\dots,n_2 \mp x,\dots,n_{\sigma(p+1)})}_A(n_1 \pm x)\\
& + \underbrace{m(n_{\sigma(1)},\dots,n_1 \pm x,n_{\sigma(j+2)},\dots,n_2 \mp x+n_{\sigma(j+1)},\dots,n_{\sigma(p+1)})}_B(n_2 \mp x))\\
=&s(n_{\sigma(j+1)})\Delta_x^m|_{x=0}(B(n_1 + n_2) + (A - B)(n_1 \pm x))\\
=&s(n_{\sigma(j+1)})((n_1 + n_2)\Delta_x^m|_{x=0}B + n_1\Delta_x^m|_{x=0}(A - B) \pm m\Delta_x^{m-1}|_{x=1}(A - B)).
\end{align*}
Then $s(n_{\sigma(j+1)})$ is bounded, $|n_1| \le N$ and $|n_1+n_2| < pn$.
If $\delta_m'$ is small enough then $|\Delta_x^m|_{x=0}B| \lesssim_m N^{\mu-m}n^{\nu(m)}$. Expanding $A-B$ into differences as in (\ref{A-B}) we know that $\Delta_x^{m-l}|_{x=l}(A - B) \lesssim_m N^{\mu-m+l-1}n^{\nu(m-l+1)+1}$. Hence $|$the $m$-th difference$| \lesssim_m N^{\mu-m}n^{\nu(m+1)+1}$. With that the proof is complete.
\end{proof}

\section{Long-time wellposedness}
Now we study the long-time wellposedness of the equation
\[
f_t = -Sf_\alpha + 2f_\alpha Sf - fSf_\alpha.
\]

First of all we notice that since $f_0$ is zero mean and with $m-$fold symmetry, for $m\geq 3$, i.e.,
$$f_0(\alpha)=\sum_{|n|\geq 1} \hat{f}_0(m n)e^{im n\alpha},$$
and the equation conserves the mean we also have that
$$f(\alpha,t)=\sum_{|n|\geq 1}\hat{f}(m n,t)e^{imn\alpha}.$$
Thus $\hat{f}(0,t)=\hat{f}(\pm 1,t)=\hat{f}(\pm 2,t)=0.$

\subsection{Energy estimates}
Define the energy
\[
E_s(f) = \frac{1}{2}\|f\|_{H^s}^2.
\]
Then
\begin{align*}
\frac{d}{dt}E_s(f) &= \langle\Lambda^sf_t, \Lambda^sf\rangle
= \langle\Lambda^s(2f_\alpha Sf - fSf_\alpha), \Lambda^sf\rangle\\
&= \langle-\Lambda^sSf_\alpha + 2Sf\Lambda^sf_\alpha + 2[\Lambda^s, Sf]f_\alpha - \Lambda^s(fSf_\alpha), \Lambda^sf\rangle.
\end{align*}
The first term vanishes because $S\partial_\alpha$ is anti-self-adjoint and commutes with $\Lambda$. The last term
\[
\langle\Lambda^s(fSf_\alpha), \Lambda^sf\rangle \in M_{3-}^{2s,0}
\]
because the multiplier of $S\partial_\alpha$ is odd and of order 0. The second term
\[
\langle Sf\Lambda^sf_\alpha, \Lambda^sf\rangle
= \frac{1}{2}\int Sf((\Lambda^sf)^2)_\alpha
= -\frac{1}{2}\int Sf_\alpha(\Lambda^sf)^2
\in M_{3-}^{2s,0}
\]
for the same reason. We show that the remaining term (the one with the commutator) is also in this class. The multiplier of this term is
\[
c(\langle n_3 \rangle^s - \langle n_2 \rangle^s)s(n_1)n_2\langle n_3 \rangle^s
\]
for some constant $c \in \C$. Since $s^{(m)}(x) \lesssim_m |x|^{-m-1}$ when $|x|$ is sufficiently large, we have the desired bound if $|n_1| > \delta N$ (Recall $N = \max(|n_1|, |n_2|, |n_3|)$.) for some $\delta > 0$.
If $|n_1| < \delta N$ and $\delta$ is small enough then
\[
\langle n_3 \rangle^s - \langle n_2 \rangle^s
= \langle n_1+n_2 \rangle^s - \langle n_2 \rangle^s
\lesssim_s n_1\langle n_2 \rangle^{s-1}
\]
and
\[
\Delta_x^m|_{x=0}(\langle n_3 \mp x \rangle^s - \langle n_2 \pm x \rangle^s)
\lesssim_{s,m} n_1\langle n_2 \rangle^{s-m-1}
\]
so the desired bound also holds, and the desired parity is easily seen.
Hence the evolution of the energy is
\begin{equation}\label{dt-Es}
\frac{d}{dt}E_s(f) = M_3(f,f,f)
\end{equation}
for some $M_3 \in M_{3-}^{2s,0}$.

\subsection{Iterated normal form transformations and proof of Theorem \ref{longexistence}}
Now we perform iterated normal form transformations on the equation (\ref{dt-Es}). By Lemma \ref{normal-form-35},
there is $M_3' = iL(M_3) \in M_{3+}^{2s,0}$ such that
\begin{align*}
\frac{d}{dt}M_3'(f,f,f)
&= M_3(f,f,f) + \sum_{j=1}^3 M_3'(f,\dots,\underbrace{2f_\alpha Sf - fSf_\alpha}_{j\text{-th entry}},\dots,f)\\
&= M_3(f,f,f) + (2N_2 - N_1)(M_3')(f,f,f,f).
\end{align*}
Then by Lemma \ref{non-lin-bound},
\[
\frac{d}{dt}(E_s(f) - M_3'(f,f,f)) = M_4(f,f,f,f)
\]
for some $M_4 \in M_{4-}^{2s,1}$. Before proceeding, we must isolate the resonance set from $M_4$. We decompose $M_4 = (M_4 - P(M_4)) + P(M_4)$,
where both parts are in $M_{4-}^{2s,1}$ by Lemma \ref{normal-form-4} (i).
By (ii), we can replace $M_4$ by $M_4 - P(M_4)$, so we assume $P(M_4) = 0$.
By (iii), there is $M_4' = iL(M_4) \in M_{4+}^{2s,4m+5}$ such that
\[
\frac{d}{dt}M_4'(f,f,f,f) = M_4(f,f,f,f) + (2N_2 - N_1)(M_4')(f,\dots,f).
\]
Then by Lemma \ref{non-lin-bound},
\[
\frac{d}{dt}(E_s(f) - M_3'(f,f,f) - M_4'(f,f,f,f)) = M_5(f,\dots,f)
\]
for some $M_5 \in M_{5-}^{2s,4m+10}$. By Lemma \ref{normal-form-35},
there is $M_5' = iL(M_5) \in M_{5+}^{2s,4m+10}$ such that
\[
\frac{d}{dt}M_5'(f,\dots,f) = M_5(f,\dots,f) + (2N_2 - N_1)(M_5')(f,\dots,f).
\]
Then by Lemma \ref{non-lin-bound},
\[
\frac{d}{dt}(E_s(f) - M_3'(f,f,f) - M_4'(f,f,f,f) - M_5'(f,\dots,f))
= M_6(f,\dots,f)
\]
for some $M_6 \in M_{6-}^{2s,4m+15}$. By Lemma \ref{mult-bound} we have, for $s \ge 16$,
\[
|M_j'(f,\dots,f)| \lesssim_s \|f\|_{H^{16}}^{j-2}\|f\|_{H^s}^2, j = 3, 4, 5
\text{ and }|M_6(f,\dots,f)| \lesssim_s \|f\|_{H^{16}}^4\|f\|_{H^s}^2.
\]
Hence there is $C = C(s) \ge 1$ such that
\begin{align}
\nonumber
\|f(t)\|_{H^s}^2
&\le \|f(0)\|_{H^s}^2 + C(\|f(0)\|_{H^{16}} + \|f(0)\|_{H^{16}}^3)\|f(0)\|_{H^s}^2\\
&+ C(\|f(t)\|_{H^{16}} + \|f(t)\|_{H^{16}}^3)\|f(t)\|_{H^s}^2
+ Ct\sup_{\tau\in[0,t]}\|f(\tau)\|_{H^{16}}^4\|f(\tau)\|_{H^s}^2.
\label{Hs-growth}
\end{align}
Thus from
\[
\|f(0)\|_{H^s} \le \ep\text{ and }\sup_{\tau\in[0,t]}\|f(\tau)\|_{H^s} \le 2\ep
\]
it follows that
\[
\sup_{\tau\in[0,t]}\|f(\tau)\|^2_{H^s} \le \ep^2 + C(9\ep + 33\ep^3)\ep^2 + 64Ct\ep^6.
\]
If $\ep \le 1/(9C)$ and $t \le 1/(64C\ep^4)$ then
\[
\sup_{\tau\in[0,t]}\|f(\tau)\|^2_{H^s} < 7\ep^2/2
\]
closing the estimate. Therefore the lifespan $\gtrsim_s 1/\ep^4$.
Thus, if the initial data is in $H^{16}$, with norm $\ep$,
then it will remain in $H^{16}$ for a period of length $\gtrsim 1/\ep^4$.
Moreover, the bound (\ref{Hs-growth}) implies that for $s \ge 16$,
the growth of $\|f\|_{H^s}$ only depends on $\|f\|_{H^{16}}$.
Hence it follows that if the initial data is in $H^s$ for $s \ge 16$,
and is sufficiently small in $H^{16}$, then it will remain in $H^s$ for a period of length $\gtrsim 1/\ep^4$, with an implicit constant independent of $s$, which is Theorem \ref{longexistence}.

\section{Analytic travelling waves}
In this section we study travelling wave solutions of the equation
\[
f_t = -Sf_\alpha + 2f_\alpha Sf - fSf_\alpha,
\]
that is, solutions of the form
\[
f(\alpha, t) = u(\alpha - vt),\quad v \in \R.
\]
For such solutions we have $f_t = -vu'$, $f_\alpha = u'$ and $Sf = Su$, so
\begin{equation}\label{u-eqn}
-Su' + vu' + 2u'Su - uSu' = 0.
\end{equation}
Clearly $u = 0$ is a solution. We will use the Crandall-Rabinowitz bifurcation theorem \cite{Crandall-Rabinowitz:bifurcation-simple-eigenvalues} to find other solutions bifurcating from the zero solution, and then show that they are analytic in $\alpha$.

\subsection{The bifurcation theorem}
%
%
%

\begin{thm}[Crandall--Rabinowitz]
Let $X$ and $Y$ be Banach spaces, $V$ a neighborhood of 0 in $X$, and
\begin{align*}
F: V \times (-1, 1) &\to Y\\
(u, \mu) &\mapsto F(u, \mu)
\end{align*}
satisfy
\begin{enumerate}[(i)]
\item $F(0, \mu) = 0$ for all $|\mu| < 1$;

\item The partial derivatives $\partial_\mu F$, $\partial_uF$ and $\partial_{u\mu}^2F$ exist and are continuous;

\item $\ker \partial_uF(0, 0) = \R u_0$ and $Y/\Im \partial_uF(0, 0)$ are one-dimensional;

\item $\partial_{u\mu}^2F(0, 0)u_0 \notin \Im \partial_uF(0, 0)$.
\end{enumerate}

Let $Z$ be a complement of $\ker \partial_uF(0, 0)$ in $X$.
Then there is a neighborhood $U$ of $(0, 0)$ in $\R \times X$,
a number $a > 0$, and continuous functions
\begin{align*}
\phi: (-a, a) &\to \R, & \psi: (-a, a) &\to Z
\end{align*}
such that $\phi(0) = 0$, $\psi(0) = 0$ and
\[
F^{-1}(0) \cap U = \{(\xi u_0 + \xi\psi(\xi), \phi(\xi)): |\xi| < a\} \cup ((\R \times \{0\}) \cap U).
\]
\end{thm}
%
%
%
%
%

\subsection{Bifurcation analysis and proof of Theorem \ref{thmtravelling}}
To apply the Crandall-Rabinowitz theorem \cite{Crandall-Rabinowitz:bifurcation-simple-eigenvalues}, we first find the linearization of the equation (\ref{u-eqn}) around the zero solution, which is
\begin{equation}\label{linearization}
-Su' + vu' = 0.
\end{equation}
Since the multiplier of $S\partial_\alpha$ is $i\lambda$, which is odd,
for $m \ge 3$ we have $S(\cos m\alpha)' = -\lambda(m)\sin m\alpha$ and $S(\sin m\alpha)' = \lambda(m)\cos m\alpha$. Thus $\cos m\alpha$ solves (\ref{linearization}) with $\lambda(m) - mv_m = 0$, i.e.,
$v_m = \lambda(m)/m$,
and $\sin m\alpha$ solves (\ref{linearization}) with the same $v_m$.
These are the only solutions of (\ref{linearization}).
Indeed, taking the Fourier transform of (\ref{linearization}) gives
\[
-i\lambda(m)\hat u(m) + imv\hat u(m) = 0
\]
so $\hat u(m) = 0$ unless $v = v_m$, which can only hold for a pair of opposite values of $m$. Since $u$ is real-valued, the two modes combines to give $\cos m\alpha$ or $\sin m\alpha$.

To perform bifurcation, fix an integer $k\ge1$. For $m\ge3$ and $c>0$ define the spaces
\begin{align*}
C^\omega_{m,c} &= \{2\pi/m\text{-periodic holomorphic functions in the strip $|\Im\alpha| < c$}\},\\
H^{k,\pm}_{m,c} &= \left\{ u \in C^\omega_c: \sup_{|y|<c} \|u(\cdot + iy)\|_{H^k(\R/2\pi\Z)} < \infty, \int_{\R/2\pi\Z} u(x + iy)dx = 0 \right.,\\
&u(-x + iy) = \pm u(x + iy),\ \forall x \in \R/2\pi\Z, |y| < c\}.
\end{align*}
with the norm (which is the same for both $H^{k,+}_{m,c}$ and $H^{k,-}_{m,c}$)
\[
\|u\|_{H^k_{m,c}} = \sup_{|y|<c} \|u(\cdot + iy)\|_{H^k(\R/2\pi\Z)}
\approx_k \|e^{c|n|}\langle n \rangle^k\|_{\ell^2_n}.
\]
Since $\partial_\alpha$ and $S$ act the same way on $u(\cdot + iy)$ for all $|y| < c$, we have that
\begin{align*}
\partial_\alpha: H^{k,+}_{m,c} &\to H^{k-1,-}_{m,c}, &
S: H^{k,\pm}_{m,c} &\to H^{k+1,\pm}_{m,c}.
\end{align*}
By the Sobolev multiplication theorem, parity considerations,
pointwise multiplication is bounded on
\begin{align*}
\times: H^{k,+}_{m,c} \times H^{k,-}_{m,c} &\to H^{k,-}_{m,c}, &
\times: H^{k-1,-}_{m,c} \times H^{k+1,+}_{m,c} &\to H^{k-1,-}_{m,c}.
\end{align*}
Hence the map
\begin{align*}
F_m: (u, \mu) &\mapsto -Su' + (v_m + \mu)u' + 2u'Su - uSu',\\
H^{k,+}_{m,c} \times \R &\mapsto H^{k-1,-}_{m,c}
\end{align*}
satisfies

\begin{enumerate}[(i)]

\item $F_m(0, \mu) = 0$ for all $\mu$;

\item The partial derivatives $\partial_\mu F_m(u, \mu) = u': H^{k,+}_{m,c} \times \R \to H^{k-1,-}_{m,c}$,
\begin{align*}
\partial_uF_m(u, \mu): w &\mapsto -Sw' + (v_m + \mu)w' + 2w'Su + 2u'Sw - wSu' - uSw',\\
\partial^2_{u\mu}F_m(u, \mu): w &\mapsto w',\\
H^{k,+}_{m,c} \times \R \times H^{k,+}_{m,c} &\to H^{k-1,-}_{m,c},
\end{align*}
all exist and are continuous.

\item $\partial_uF_m(0, 0)(w) = -Sw' + v_mw'$ so $\ker \partial_uF_m(0, 0) = \R\cos m\alpha$ and
\[
W_{m,c} := \Im \partial_uF_m(0, 0) = \{u \in H^{k-1,-}_{m,c}: \hat u(m) = 0\}
\]
has codimension one.

\item $\partial^2_{u\mu}F_m(0, 0)(\cos m\alpha) = -m\sin m\alpha \notin W_{m,c}$.

\end{enumerate}

Let
\[
Z_{m,c} = \{u \in H^{k,+}_{m,c}: \hat u(m) = 0\}
\]
be a complement of $\R\cos n\alpha$ in $H^{k,+}_{m,c}$. Then the Crandall--Rabinowitz theorem shows that there are a number $\ep_{m,c} > 0$,
an open interval $I_{m,c}$ containing 0 and continuous functions
\[
(\phi_{m,c}, \psi_{m,c}): I_{m,c} \to \R \times Z_{m,c}
\]
such that $\phi_{m,c}(0) = 0$, $\psi_{m,c}(0) = 0$ and
\[
F_m^{-1}(0) \cap B_{H_{m,c}^{k,+} \times \R}(\ep_{m,c}) = \{(\xi\cos m\alpha + \xi\psi_{m,c}(\xi), \phi_{m,c}(\xi)),\ \xi \in I_{m,c}\}.
\]
We can remove the dependence of $c$ as follows: Let $c' > c$.
Then the above produces an $\ep_{m,c'} > 0$.
Let $I$ be an open subinterval of $I_{m,c}$ such that
\[
(\xi\cos m\alpha + \xi\psi_{m,c}(\xi), \phi_{m,c}(\xi))
\in F_m^{-1}(0) \cap B_{H_{m,c}^{k,+} \times \R}(\ep_{m,c})
\cap B_{H_{m,c'}^{k,+} \times \R}(\ep_{m,c'})
\]
for all $\xi \in I$. Then for each $\xi \in I \backslash\{0\}$ we have that
\[
(\xi\cos m\alpha + \xi\psi_{m,c}(\xi), \phi_{m,c}(\xi))
= (\xi'\cos m\alpha + \xi'\psi_{m,c'}(\xi'), \phi_{m,c'}(\xi'))
\]
for some $\xi' \in I_{m,c'} \backslash \{0\}$. Since $H^{k,+}_{m,c} = Z_{m,c} \oplus \R\cos m\alpha$ is a direct sum, $\xi = \xi'$ and $\psi_{m,c}(\xi) = \psi_{m,c'}(\xi) \in Z_{n,c'}$ (here $\xi \neq 0$ is used).
We also have $\phi_{m,c}(\xi) = \phi_{m,c'}(\xi)$.
Hence we can omit $c$ from the subscripts of $\psi$ and $\phi$. Now let
\[
u_{m,\xi}(\alpha): = \xi\cos m\alpha + \xi\psi_m(\xi).
\]
Then
\[
f_{m,\xi}(\alpha, t) := u_{m,\xi}(\alpha - (v_m + \phi_m(\xi))t)
\]
is a travelling wave moving at the velocity $v_m + \phi_m(\xi)$.
Moreover for any $c > 0$ there is an open interval $I$ containing 0 such that for all $\xi \in I$, the travelling waves $f_{m,\xi}$ are analytic in the strip $\{|\Im\alpha| < c\}$. This proves Theorem \ref{thmtravelling}.

\section*{Acknowledgements}
 This work is supported in part by the Spanish Ministry of Economy under the ICMAT Severo Ochoa grant SEV2015-0554 and MTM2017-89976-P. AC was partially supported by the Europa Excelencia program ERC2018-092824. DC and FZ were partially supported by the ERC Advanced Grant 788250.

\bibliographystyle{abbrv}
\bibliography{references}

\def\cprime{$'$}
\begin{thebibliography}{10}

\bibitem{Ar}
V.~I. Arnol\cprime~d.
\newblock {\em Geometrical methods in the theory of ordinary differential
  equations}, volume 250 of {\em Grundlehren der Mathematischen Wissenschaften
  [Fundamental Principles of Mathematical Science]}.
\newblock Springer-Verlag, New York-Berlin, 1983.
\newblock Translated from the Russian by Joseph Sz\"{u}cs, Translation edited
  by Mark Levi.

\bibitem{TSV19}
T.~Buckmaster, S.~Shkoller, and V.~Vicol.
\newblock Nonuniqueness of weak solutions to the {SQG} equation.
\newblock {\em Comm. Pure Appl. Math.}, 72(9):1809--1874, 2019.

\bibitem{CaCo10}
A.~Castro and D.~C\'{o}rdoba.
\newblock Infinite energy solutions of the surface quasi-geostrophic equation.
\newblock {\em Adv. Math.}, 225(4):1820--1829, 2010.

\bibitem{Castro-Cordoba-GomezSerrano:global-smooth-solutions-sqg}
A.~Castro, D.~C{\'o}rdoba, and J.~G{\'o}mez-Serrano.
\newblock Global smooth solutions for the inviscid {S}{Q}{G} equation.
\newblock {\em Arxiv preprint arXiv:1603.03325}, 2016. To appear in Memoirs of
  the AMS.

\bibitem{Const94}
P.~Constantin.
\newblock Geometric statistics in turbulence.
\newblock {\em SIAM Rev.}, 36(1):73--98, 1994.

\bibitem{Constantin-Lai-Sharma-Tseng-Wu:new-numerics-sqg}
P.~Constantin, M.-C. Lai, R.~Sharma, Y.-H. Tseng, and J.~Wu.
\newblock New numerical results for the surface quasi-geostrophic equation.
\newblock {\em J. Sci. Comput.}, 50(1):1--28, 2012.

\bibitem{Constantin-Majda-Tabak:formation-fronts-qg}
P.~Constantin, A.~J. Majda, and E.~Tabak.
\newblock Formation of strong fronts in the {$2$}-{D} quasigeostrophic thermal
  active scalar.
\newblock {\em Nonlinearity}, 7(6):1495--1533, 1994.

\bibitem{ConstNguy18}
P.~Constantin and H.~Q. Nguyen.
\newblock Global weak solutions for {SQG} in bounded domains.
\newblock {\em Comm. Pure Appl. Math.}, 71(11):2323--2333, 2018.

\bibitem{Constantin-Nie-Schorghofer:nonsingular-sqg-flow}
P.~Constantin, Q.~Nie, and N.~Sch{\"o}rghofer.
\newblock Nonsingular surface quasi-geostrophic flow.
\newblock {\em Phys. Lett. A}, 241(3):168--172, 1998.

\bibitem{Cordoba:nonexistence-hyperbolic-blowup-qg}
D.~C\'ordoba.
\newblock Nonexistence of simple hyperbolic blow-up for the quasi-geostrophic
  equation.
\newblock {\em Ann. of Math. (2)}, 148(3):1135--1152, 1998.

\bibitem{Cordoba-Fefferman:growth-solutions-qg-2d-euler}
D.~C\'ordoba and C.~Fefferman.
\newblock Growth of solutions for {QG} and 2{D} {E}uler equations.
\newblock {\em J. Amer. Math. Soc.}, 15(3):665--670, 2002.

\bibitem{CGI}
D.~C\'{o}rdoba, J.~G\'{o}mez-Serrano, and A.~D. Ionescu.
\newblock Global solutions for the generalized {SQG} patch equation.
\newblock {\em Arch. Ration. Mech. Anal.}, 233(3):1211--1251, 2019.

\bibitem{Crandall-Rabinowitz:bifurcation-simple-eigenvalues}
M.~G. Crandall and P.~H. Rabinowitz.
\newblock Bifurcation from simple eigenvalues.
\newblock {\em J. Functional Analysis}, 8:321--340, 1971.

\bibitem{Dritschel:exact-rotating-solution-sqg}
D.~G. Dritschel.
\newblock An exact steadily rotating surface quasi-geostrophic elliptical
  vortex.
\newblock {\em Geophys. Astrophys. Fluid Dyn.}, 105(4-5):368--376, 2011.

\bibitem{EJ20}
T.~M. Elgindi and I.-J. Jeong.
\newblock Symmetries and critical phenomena in fluids.
\newblock {\em Comm. Pure Appl. Math.}, 73(2):257--316, 2020.

\bibitem{Friedlander-Shvydkoy:unstable-spectrum-sqg}
S.~Friedlander and R.~Shvydkoy.
\newblock The unstable spectrum of the surface quasi-geostropic equation.
\newblock {\em J. Math. Fluid Mech.}, 7(suppl. 1):S81--S93, 2005.

\bibitem{gancedo2018local}
F.~Gancedo and N.~Patel.
\newblock On the local existence and blow-up for generalized sqg patches, 2018.

\bibitem{GS19}
P.~Gravejat and D.~Smets.
\newblock Smooth travelling-wave solutions to the inviscid surface
  quasi-geostrophic equation.
\newblock {\em Int. Math. Res. Not. IMRN}, (6):1744--1757, 2019.

\bibitem{Held-Pierrehumbert-Garner-Swanson:sqg-dynamics}
I.~M. Held, R.~T. Pierrehumbert, S.~T. Garner, and K.~L. Swanson.
\newblock Surface quasi-geostrophic dynamics.
\newblock {\em J. Fluid Mech.}, 282:1--20, 1995.

\bibitem{hunter2018global}
J.~K. Hunter, J.~Shu, and Q.~Zhang.
\newblock Global solutions of a surface quasi-geostrophic front equation, 2018.

\bibitem{hunter2020global}
J.~K. Hunter, J.~Shu, and Q.~Zhang.
\newblock Global solutions for a family of gsqg front equations, 2020.

\bibitem{KaPo}
T.~Kappeler and J.~P\"{o}schel.
\newblock {\em Kd{V} \& {KAM}}, volume~45 of {\em Ergebnisse der Mathematik und
  ihrer Grenzgebiete. 3. Folge. A Series of Modern Surveys in Mathematics
  [Results in Mathematics and Related Areas. 3rd Series. A Series of Modern
  Surveys in Mathematics]}.
\newblock Springer-Verlag, Berlin, 2003.

\bibitem{KN12}
A.~Kiselev and F.~Nazarov.
\newblock A simple energy pump for the surface quasi-geostrophic equation.
\newblock 7:175--179, 2012.

\bibitem{KRYZ}
A.~Kiselev, L.~Ryzhik, Y.~Yao, and A.~Zlato\v{s}.
\newblock Finite time singularity for the modified {SQG} patch equation.
\newblock {\em Ann. of Math. (2)}, 184(3):909--948, 2016.

\bibitem{BeFePu}
R.~F. M.~Berti and F.~Pusateri.
\newblock Birkhoff normal form and long time existence for periodic gravity
  water waves.
\newblock {\em Preprint arXiv:1810.11549.}

\bibitem{Majda-Bertozzi:vorticity-incompressible-flow}
A.~J. Majda and A.~L. Bertozzi.
\newblock {\em Vorticity and incompressible flow}, volume~27 of {\em Cambridge
  Texts in Applied Mathematics}.
\newblock Cambridge University Press, Cambridge, 2002.

\bibitem{Marchand:existence-regularity-weak-solutions-sqg}
F.~Marchand.
\newblock Existence and regularity of weak solutions to the quasi-geostrophic
  equations in the spaces {$L^p$} or {$\dot H^{-1/2}$}.
\newblock {\em Comm. Math. Phys.}, 277(1):45--67, 2008.

\bibitem{Ohkitani-Yamada:inviscid-limit-sqg}
K.~Ohkitani and M.~Yamada.
\newblock Inviscid and inviscid-limit behavior of a surface quasigeostrophic
  flow.
\newblock {\em Phys. Fluids}, 9(4):876--882, 1997.

\bibitem{Pedlosky:geophysical}
J.~Pedlosky.
\newblock Geophysical fluid dynamics.
\newblock {\em New York and Berlin, Springer-Verlag}, 1, 1982.

\bibitem{Resnick:phd-thesis-sqg-chicago}
S.~G. Resnick.
\newblock {\em Dynamical problems in non-linear advective partial differential
  equations}.
\newblock PhD thesis, University of Chicago, Department of Mathematics, 1995.

\bibitem{Sh}
J.~Shatah.
\newblock Normal forms and quadratic nonlinear {K}lein-{G}ordon equations.
\newblock {\em Comm. Pure Appl. Math.}, 38(5):685--696, 1985.

\end{thebibliography}






\begin{tabular}{l}
\textbf{Angel Castro} \\
 {\small Instituto de Ciencias Matematicas-CSIC-UAM-UC3M-UCM}\\
 {\small Consejo Superior de Investigaciones Cientificas} \\
 {\small C/ Nicolas Cabrera, 13-15, 28049 Madrid, Spain} \\
  {\small Email: angel\_castro@icmat.es} \\
\\
\textbf{Diego Cordoba} \\
  {\small Instituto de Ciencias Matematicas-CSIC-UAM-UC3M-UCM} \\
 {\small Consejo Superior de Investigaciones Cientificas} \\
 {\small C/ Nicolas Cabrera, 13-15, 28049 Madrid, Spain} \\
  {\small Email: dcg@icmat.es} \\
\\
\textbf{Fan Zheng} \\
 {\small Instituto de Ciencias Matematicas-CSIC-UAM-UC3M-UCM} \\
 {\small Consejo Superior de Investigaciones Cientificas} \\
 {\small C/ Nicolas Cabrera, 13-15, 28049 Madrid, Spain} \\
  {\small Email: fan.zheng@icmat.es} \\
  \\

\end{tabular}
\end{document}